\documentclass[12pt,reqno]{amsart}
\usepackage{mathrsfs}
\usepackage{amsgen}
\usepackage{amscd}
\usepackage{amsmath}
\usepackage{latexsym}
\usepackage{amsfonts}
\usepackage{amssymb}
\usepackage{amsthm}
\usepackage{graphicx}
\usepackage{color}
\usepackage{amsthm,amssymb}
\usepackage{graphicx}
\usepackage{enumerate}
\usepackage{amssymb,amsmath,graphicx,amsfonts,euscript}
\usepackage{color}
\usepackage{mathrsfs}
\usepackage{empheq}
\usepackage{cases}
\usepackage{cite}

\usepackage[margin=3cm, a4paper]{geometry}

\def\H{{\cal H}}

\def\N{\mathbb{N}}
\def\R{\mathbb{R}}

\def\T{\mathbb{T}}
\def\C{\mathbb{C}}

\def\H2{H^2(\R^N)}
\def\L2{L^2(\R^N)}

 \def\norm#1{\|#1\|}

\def\H{{\cal H}}

\def\H1{H^1(\R)}

\newcommand{\wt}{\widetilde}

\newcommand{\De}{\Delta}

   \newcommand{\I}{\infty}

\newcommand{\EQ}[1]{\begin{equation}\begin{split}
      #1 \end{split}\end{equation}} 
 \newcommand{\Del}[1]{}

\numberwithin{equation}{section}

\newtheorem{thm}{Theorem}[section]
\newtheorem{cor}[thm]{Corollary}
\newtheorem{lem}[thm]{Lemma}
\newtheorem{prop}[thm]{Proposition}
\newtheorem{definition}[thm]{Definition}
\theoremstyle{remark}
\newtheorem{remark}[thm]{Remark}
\newtheorem*{exam*}{Examples}

\begin{document}

\setcounter{page}{1}

\title[GLOBAL WELL-POSEDNESS FOR ROUGH DATA]{Large global solutions for energy-critical nonlinear Schr\"{o}dinger equation}

\author{Ruobing Bai}
\address{Center for Applied Mathematics\\
Tianjin University\\
Tianjin 300072, China}
\email{baimaths@hotmail.com}
\thanks{}

\author{Jia Shen}
\address{Center for Applied Mathematics\\
Tianjin University\\
Tianjin 300072, China}
\email{shenjia@tju.edu.cn}
\thanks{}

\author{Yifei Wu}
\address{Center for Applied Mathematics\\
Tianjin University\\
Tianjin 300072, China}
\email{yerfmath@gmail.com}
\thanks{}




\maketitle

\begin{abstract}\noindent
In this work, we consider the 3D defocusing energy-critical nonlinear Schr\"odinger equation
\begin{align*}
 i\partial_t u+\Delta u =|u|^4 u, \quad (t,x)\in \R\times \R^3.
\end{align*}
Applying the outgoing and incoming decomposition presented in the recent work \cite{BECEANU-DENG-SOFFER-WU-2021}, we prove that any radial function $f$ with $\chi_{\leq1}f\in H^1$ and $\chi_{\geq1}f\in H^{s_0}$ with $\frac{5}{6}<s_0<1$, there exists an outgoing component $f_+$ (or incoming component $f_-$) of $f$, such that when the initial data is $f_+$, then the corresponding solution is globally well-posed and scatters forward in time; when the initial data is $f_-$, then the corresponding solution is globally well-posed and scatters backward in time. 
\end{abstract}

\vskip 0.2cm

\section{Introduction}
In this paper, we consider the Cauchy problem  for the nonlinear Schr\"{o}dinger equation (NLS) in 3 spatial dimensions (3D)
\begin{equation}\label{NLS}
   \left\{ \aligned
    &i\partial_t u+\Delta u =\mu |u|^4 u,
    \\
    &u(0,x)=u_0(x),
   \endaligned
  \right.
 \end{equation}
with $\mu=\pm 1$. Here $u=u(t,x): \R\times\R^3\rightarrow \C$ is a complex-valued function. The case $\mu=1$ is referred to the defocusing case, and the case $\mu=-1$ is referred to the focusing case.

The solution satisfies
the conservation of mass and energy, defined respectively by
\begin{align}\label{1.2}
M\big(u(t)\big):=\int_{\R^3}|u(t,x)|^2dx=M(u_0),
\end{align}
and
\begin{align}\label{1.3}
E\big(u(t)\big):=\frac12\int_{\R^3}|\nabla u(t,x)|^2dx+\frac{\mu}{6}\int_{\R^3}|u(t,x)|^{6}dx=E(u_0).
\end{align}
The general form of the equation \eqref{NLS} is the following
\begin{equation}\label{gNLS}
   \left\{ \aligned
    &i\partial_t u+\Delta u =\mu |u|^p u,\quad\quad (t, x)\in \R^{1+d},
    \\
    &u(0,x)=u_0(x).
   \endaligned
  \right.
 \end{equation}
The class of solutions to equation \eqref{gNLS} is invariant under scaling
\begin{align*}
u(t,x)\rightarrow u_{\lambda}(t,x)=\lambda^{\frac{2}{p}}u(\lambda^2t,\lambda x), \quad \lambda>0,
\end{align*}
which maps the initial data
\begin{align*}
u(0)\rightarrow u_\lambda(0):=\lambda^{\frac{2}{p}}u_0(\lambda x).
\end{align*}
Denote $$s_c=\frac{d}{2}-\frac{2}{p},$$
then the scaling leaves $\dot{H}^{s_c}$ norm invariant, that is,
\[
\|u(0)\|_{\dot{H}^{s_c}}=
\|u_{\lambda}(0)\|_{\dot{H}^{s_c}}.
\]

The well-posedness and scattering theory for the equation \eqref{gNLS} has been widely studied. For the local well-posedness, Cazenave and Weissler \cite{Cazenave-Weissler-1990} used the standard fixed point argument, and proved the equation \eqref{gNLS} is locally well-posed in $H^{s}(\R^d)$ when $s\geq s_c$. Note that in the case of $s=s_c$ (critical regime), the time of existence depends on the profile of initial data rather than simply its norm. The fixed point argument can also be applied directly to prove the global well-posedness and scattering for the equation \eqref{gNLS} with small initial data in $H^{s}(\R^d)$ when $s\geq s_c$.

Next, let us briefly review the large data global well-posedness and scattering theory for  energy-critical NLS \eqref{gNLS}. Bourgain \cite{Bourgain-1999} firstly obtained such result for the 3D and 4D defocusing energy critical NLS with radial data in $\dot{H}^{1}(\R^3)$ by introducing the induction on energy method and spatial localized Morawetz estimate. Moreover, Grillakis \cite{Grillakis-2000} provided a different proof for the global well-posedness part of the result by Bourgain \cite{Bourgain-1999}. Tao \cite{Tao-2005} later generalized the results in \cite{Bourgain-1999,Grillakis-2000} to general dimensions with radially symmetric data. For non-radial problem, a major breakthrough was made by Colliander, Keel, Staffilani, Takaoka, and Tao in \cite{Colliander-Keel-Staffilani-Takaoka-Tao-2008}, where they obtained the related result for the 3D energy-critical defocusing NLS for general large data in $\dot{H}^{1}$. Then, the result was generalized by Ryckman and Visan \cite{Ryckman-Visan-2007} in dimension $d=4$ and Visan \cite{Visan-2007} for higher dimensions. In the focusing case, Kenig and Merle \cite{Kenig-Merle-2006} introduced the concentration compactness method, and obtained the global well-posedness and scattering in $\dot{H}^{1}(\R^d)$ ($d=3,4,5$) for the energy-critical NLS with radial initial data below the energy of ground state. Killip and Visan \cite{Killip-Visan-2010} later obtained the related result for dimensions $d\geq5$ without the radial assumption. Then, Donson \cite{Dodson-2019} solved the 4D non-radial problem. Here, we only mention the papers for energy critical equations \eqref{gNLS}, and some other results for \eqref{gNLS} can be found in \cite{Dodson-2012, Dodson-2016-1, Dodson-2016-2, Dodson-Miao-Murphy-Zheng-2017, BECEANU-DENG-SOFFER-WU-2021-2, BECEANU-DENG-SOFFER-WU-2021-3, Kenig-Merle-2010, Miao-Murphy-Zheng-2014, Shen-Soffer-Wu-2021, Shen-Soffer-Wu-2021-2} and the references therein.



Although the equation \eqref{gNLS} is ill-posed in super-critical spaces, there are still some methods to study the well-posedness for a class of such data. The ill-posedness in some cases can be circumvented by an appropriate probabilistic method. Bourgain \cite{Bourgain-1994, Bourgain-1996} obtained the first almost sure local and global well-posedness results, which are based on the invariance of Gibbs measure associated to NLS on torus in one and two space dimensions. The random data approach has been further developed for the nonlinear dispersive equations, see for example \cite{Burq-Tzvetkov-2008-1, Burq-Tzvetkov-2008-2, Benyi-Oh-Pocovnicu-2015, Colliander-Oh-2012, Dodson-Luhrmann-Mendelson-2019, Killip-Murphy-Visan-2019} and the references therein.


\subsection{Main result}
This paper aims to consider the global-wellposedness and scattering of the energy-critical NLS with rough and determined initial data. This is a continuing work of \cite{BECEANU-DENG-SOFFER-WU-2021}, 
where the authors constructed the incoming and outgoing waves for the linear Schr\"{o}dinger flow and obtained global well-posedness and scattering in the inter-critical case with suitable rough data (the part near the origin of the initial data belongs to ${H}^1$, and the part away from the origin is rough).

Next, we recall the the definitions of the incoming and outgoing components of functions introduced in \cite{BECEANU-DENG-SOFFER-WU-2021}.
\begin{definition}[Deformed Fourier transformation]
Let $\alpha\in \R$, $\beta \in \R$, and let $f\in S(\R^d)$ with $|x|^\beta f \in L_{loc}^1(\R^d)$. Define
\begin{align*}
\mathcal{F}f(\xi)=|\xi|^\alpha \int_{\R^d}e^{-2\pi ix\cdot\xi}|x|^\beta f(x)dx.
\end{align*}
\end{definition}
\begin{definition}[Definitions of outgoing and incoming components]
Let $\alpha<3$, $\beta>-3$, and the function $f\in L_{loc}^1(\R^d)$ is radial.
Define the outgoing component of $f$ as
\begin{align*}
f_{out}(r)=r^{-\beta}\int_0^{+\infty}\big(J(\rho r)-K(\rho r)\big)\rho^{-\alpha+2}\mathcal{F}f(\rho)d\rho,
\end{align*}
the incoming component of $f$ as
\begin{align*}
f_{in}(r)=r^{-\beta}\int_0^{+\infty}\big(J(-\rho r)+K(\rho r)\big)\rho^{-\alpha+2}\mathcal{F}f(\rho)d\rho,
\end{align*}
where
\begin{align*}
J(r)=\int_0^{\frac{\pi}{2}}e^{2\pi ir sin\theta}cos\theta d\theta,
\end{align*}
and
\begin{align*}
K(r)=\chi_{\geq2}(r)\Big[-\frac{1}{2\pi ir}+\frac{d-3}{(2\pi ir)^3}\Big], \quad d=3, 4, 5.
\end{align*}
\end{definition}
\begin{definition}[Definitions of modified outgoing and incoming components]
Let the radial function $f\in S(\R^d)$. Define the modified outgoing component of $f$ as
\begin{align*}
f_+=\frac{1}{2}P_{\leq 1}f+\frac{1}{2}P_{\geq 1} \chi_{\leq \varepsilon_0}f+(P_{\geq 1} \chi_{\geq \varepsilon_0}f)_{out};
\end{align*}
the modified incoming component of $f$ as
\begin{align*}
f_-=\frac{1}{2}P_{\leq 1}f+\frac{1}{2}P_{\geq 1} \chi_{\leq \varepsilon_0}f+(P_{\geq 1} \chi_{\geq \varepsilon_0}f)_{in}.
\end{align*}
\end{definition}

The main observation in \cite{BECEANU-DENG-SOFFER-WU-2021} is that the decomposition allows us to cut the linear flow $e^{it\Delta}f$ into $e^{it\Delta}f_+$ and $e^{it\Delta}f_-$ such that up to a smooth part, the former moves forward in time and the latter moves backward in time, and the speed depends on the frequency which is faster for rougher data.
By using the decomposition, the authors obtained positively (or negatively) global when the initial data is $f_+$ (or $f_-$). From the definitions, we have $f=f_++f_-$. Since $f$ is rough, at least one of $f_+$ and $f_-$ is rough. Therefore, the authors obtained the global solutions for the defocusing energy-subcritical NLS, $p<\frac{4}{d-2}$, with a class of initial data in the supercritical space in dimensions $d=3,4,5$.


In this paper, we further consider the energy-critical case when $p=\frac{4}{d-2}$.
This is a more complex scenario, since the local lifespan depends on the profile of initial data, and the new difficulty is how to extend the solution globally. Moreover, we only consider the 3D case, and  the calculation is similar for other dimensional cases.

The following is our main result.
\begin{thm}\label{main theorem}
Let $s_0\in (\frac{5}{6}, 1)$. Suppose that $f$ is a radial function and there exists $\varepsilon_0>0$
 such that
 \begin{align*}
 \chi_{\leq\varepsilon_0}f\in H^1(\R^3), \quad(1-\chi_{\leq\varepsilon_0})f\in H^{s_0}(\R^3).
 \end{align*}
 Then the solution $u$ to the equation \eqref{NLS} with the initial data
 \begin{align*}
 u_0=f_+\quad(or\quad u_0=f_-)
 \end{align*}
 is global forward (or backward) in time. Moreover, there exists $u_+\in H^{s_0}(\R^3)$ (or $u_-\in H^{s_0}(\R^3)$), such that
 \begin{align*}
{\lim_{t\rightarrow+\infty}}\|u(t)-e^{it\Delta}u_+\|_{H^1(\R^3)}\rightarrow0 \quad(or {\lim_{t\rightarrow-\infty}}\|u(t)-e^{it\Delta}u_-\|_{H^1(\R^3)}\rightarrow0).
\end{align*}
 \end{thm}
We make several remarks regarding the above statements.
 \begin{remark}
\begin{enumerate}
\item
Note that $s_0<1$, we are able to construct the global solutions for 3D defocusing energy-critical NLS with a class of data in the supercritical space. Moreover, there is no size restriction for the initial data.
\item
By rescaling, it suffices to prove Theorem \ref{main theorem} when $\varepsilon_0=1$.
\end{enumerate}
\end{remark}

\subsection{ The key ingredients in the proofs}

In this subsection, we describe the key ingredients of the proof for Theorem \ref{main theorem}.

$\bullet$ \textit{A conditional perturbation theory.} We consider the perturbation equation
\begin{equation*}
   \left\{ \aligned
    &i\partial_t w+\Delta w =|v+w|^4(v+w),
    \\
    &w(0,x)=w_0(x)\in H^1(\R^3).
   \endaligned
  \right.
\end{equation*}
Given the maximal lifespan $[0,T^*)$, under the hypothesis of
$$
\mbox{(a):} \,\, v\in S(I); \quad
\mbox{(b):} \,\, w\in L^\infty_t([0,T^*); H^1),
$$
where $S([0,T^*))$ is some suitably defined spacetime norm at $\dot H^1$ level, see \eqref{Def:SI} below,  we establish the spacetime estimate that
$$
\|w\|_{S([0,T^*))}<+\infty.
$$
Here the bound is independent of $T^*$.
This is a general theory on  the scattering of the solution to a perturbation equation, which is available in our case by splitting $u$ into a linear part $v$ and a solution $w$ of a perturbation equation.

To prove the general theory, we will adopt the perturbation theory to approximate $w$ by the solution of the original energy-critical NLS, which is inspired by \cite{Benyi-Oh-Pocovnicu-2015, Dodson-Luhrmann-Mendelson-2019, Tao-Visan-Zhang-2007-1}. 
The key ingredient of perturbation theory is to construct a suitable auxiliary space $S(I)$ that can close the estimates for nonlinear interaction, and meanwhile it matches the smooth effect of the linear flow benefited from the incoming/outgoing decomposition.

Based on the perturbation theory, it reduces to check the two hypotheses, namely to prove some required spacetime estimates of the linear flow, and the uniform bound of perturbation equation's energy norm.

$\bullet$ \textit{Supercritical spacetime estimates of the linear flow.} In this part, we check the hypothesis (a) above. Noting that $u_0$ merely belongs to $H^{s_0}$, $\frac{5}{6}<s_0<1$, the $S(I)$-estimates for linear solution are supercritical. Hence, we need to obtain enough smoothing effect from the incoming and outgoing decomposition by the delicate phase-space analysis method in  \cite{BECEANU-DENG-SOFFER-WU-2021}. However, the estimates presented in  \cite{BECEANU-DENG-SOFFER-WU-2021} is not sufficient in our case. Thus, we prove some finer estimates for the spacetime norm. In particular, these estimates imply that the solution becomes better when the time is away from the origin, which is crucial in the proof of the a priori estimate.


$\bullet$ \textit{A priori estimate.}
In this step, we shall obtain the a priori estimate of the solution $w$ to the perturbation equation in $H^1$ while the initial data is only in $H^{s_0}$, $\frac{5}{6}< s_0<1$. We first make a decomposition of initial data, such that $w_0$ is in $H_x^1$. Then, the proof of the a priori estimate is based on the Morawetz estimates, energy estimates and bootstrap argument.
In the last step, we can obtain the $\norm{\nabla v}_{L_t^2 L_x^\I}$-estimate for the linear flow $v$, which is sufficient for controlling the energy increment. However, such estimate is only available when $t$ is away from the origin. Therefore, we consider the short time and long time cases, separately.

As for the short time interval, the estimates for linear solution $e^{it\De} u_0$ is not smooth enough, but it still enables some $\dot H_x^1$-critical spacetime estimates while the initial data $u_0$ itself is below the energy regularity. Thus, applying the local theory, we expect that the original solution $u$ also has some $\dot H_x^1$ level estimates, and then the energy increment can be bounded suitably. Finally, we remark that the lower bound $\frac56$ of the regularity condition is to ensure that the local $\dot H_x^1$-critical estimates hold.


\bigskip

The rest of the paper is organized as follows. In Section 2, we give some basic notations and lemmas that will be used throughout this paper. In Section 3, we give a general theory about the existence of the solution to a perturbation equation under suitable a priori hypothesis. In Section 4, we give the framework for proof of Theorem \ref{main theorem}. In Section 5, We obtain the uniform bounded of linear solution $v$ in the auxiliary space $Y(I)$. In Section 6, we prove the a priori estimate of solution $w$ to the perturbation equation in $H^1$.
\section{Preliminary}

\subsection{Notations}
We write $X\lesssim Y$ or $Y\gtrsim X$ to denote the estimate $X\leq CY$ for some constant $C>0$. Throughout the whole paper, the letter $C$ will denote different positive constants which are not important in our analysis and may vary line by line. If $C$ depends upon some additional parameters, we will indicate this with subscripts; for example, $X\lesssim_a Y$ denotes the $X\leq C(a)Y$ assertion for some $C(a)$ depending on $a$. The notation $a+$ denotes $a+\epsilon$ for some small $\epsilon$.
We use the following norms to denote the mixed spaces $L_t^qL_x^r$, that is
\begin{align*}
\|u\|_{L_t^qL_x^r}=\big(\int \|u\|_{L_x^r}^qdt\big)^{\frac{1}{q}}.
\end{align*}
When $q=r$ we abbreviate $L_t^qL_x^q$ as $L_{t,x}^q$.
We use $\chi_{\leq a}$ for $a\in \R^+$ to be the smooth function
$$
\chi_{\leq a}(x)=\left\{ \aligned
    &1, |x|\leq a,
    \\
    &0, |x|\geq \frac{11}{10}a.
   \endaligned
  \right.
$$
Moreover, we denote $\chi_{\geq a}=1-\chi_{\leq a}$ and $\chi_{a\leq\cdot\leq b}=\chi_{\leq b}-\chi_{\leq a}$. We denote $\chi_a=\chi_{\leq 2a}-\chi_{\leq a}$ for short.

For each number $N>0$, we define the Fourier multipliers $P_{\leq N}, P_{>N}, P_N$ as
$$
\widehat{P_{\leq N}f}(\xi):=\chi_{\leq N}(\xi)\widehat{f}(\xi),
$$
$$
\widehat{P_{> N}f}(\xi):=\chi_{> N}(\xi)\widehat{f}(\xi),
$$
$$
\widehat{P_{N}f}(\xi):=\chi_{N}(\xi)\widehat{f}(\xi),
$$
and similarly $P_{\geq N}, P_{<N}$. These multipliers are usually used when $N$ are $dyadic$  $numbers$ (that is, of the form $2^k$ for some integer $k$).

Moreover, we denote the Strichartz norm by
\begin{align*}
\|u\|_{S^0(I)}:=\sup\big\{\|u\|_{L_t^{q}L_x^{r}(I\times{\R^3})}: \frac{2}{q}+\frac{3}{r}=\frac{3}{2}, 2\leq q\leq \infty, 2\leq r \leq 6 \big\},
\end{align*}
and the dual space of $S^0(I)$ by $N^0(I)$ and the corresponding norm is
\begin{align*}
\|u\|_{N^0(I)}:=\inf\big\{\|u\|_{L_t^{q'}L_x^{r'}(I\times{\R^3})}: \frac{2}{q}+\frac{3}{r}=\frac{3}{2}, 2\leq q\leq \infty, 2\leq r \leq 6\big\}.
\end{align*}

\subsection{Basic lemmas}
In this section, we state some useful lemmas which will be used in our later sections. Firstly, we recall the well-known Strichartz estimates.
\begin{lem}\label{Strichartz estimate}
(Strichartz's estimates, see \cite{Cazenave-2003}) Let $I\subset \R$ be a time interval. For all admissible pairs $(q_j, r_j), j=1,2,$ satisfying
\begin{align*}
2\leq q_j, r_j\leq \infty  ,\quad \frac{2}{q_j}+\frac{d}{r_j}=\frac{d}{2},\quad and\quad (q,r,d)\ne(2,\I,2),
\end{align*}
then the following statements hold:
\begin{align}\label{1.2222}
\|e^{it\Delta}f\|_{L_t^{q_j}L_x^{r_j}(I\times{\R^d})}\lesssim\|f\|_{L^2(\R^d)};
\end{align}
and
\begin{align}\label{1.222234}
\Big\|\int_0^t e^{i(t-s)\Delta}F(s)ds\Big\|_{L_t^{q_1}L_x^{r_1}(I\times{\R^d})}\lesssim \|F\|_{L_t^{q_2'}L_x^{r_2'}(I\times{\R^d})},
\end{align}
where $\frac{1}{q_2}+\frac{1}{q_2'}=\frac{1}{r_2}+\frac{1}{r_2'}=1$.
\end{lem}
We also need the following radial Sobolev inequality.
\begin{lem}\label{radial Sobolev}
(see \cite{Tao-Visan-Zhang-2007})Let $\alpha, p, q ,s$ be the parameters which satisfy
\begin{align*}
\alpha >-\frac{d}{q}, \quad \frac{1}{q}\leq \frac{1}{p}\leq \frac{1}{q}+s, \quad 1\leq p,q\leq \infty, \quad 0<s<d
\end{align*}
with
$$\alpha+s=d(\frac{1}{p}-\frac{1}{q}).$$
Moreover, at most one of the equalities holds:
$$
p=1, \quad p=\infty, \quad q=1,\quad q=\infty,\quad \frac{1}{p}=\frac{1}{q}+s.
$$
Then for any radial function $u$,
\begin{align*}
\big\||x|^\alpha u\big\|_{L^q(\R^d)}\lesssim \||\nabla|^s u\|_{L^p(\R^d)}.
\end{align*}
\end{lem}
The following result is the Hardy inequality.
\begin{lem}\label{hardy}
Let $1<p<d$. Then,
$$
\big\||x|^{-1}u\big\|_{L_x^p(\R^d)}\lesssim \|\nabla u\|_{L_x^p(\R^d)}.
$$
\end{lem}
\section{A general perturbation theory}

In this section, we set up a general theory to give a sufficient condition for the existence of the solution to the following perturbation equation
\begin{equation}\label{PNLS}
   \left\{ \aligned
    &i\partial_t w+\Delta w =|w+v|^4 (w+v),
    \\
    &w(0,x)=w_0(x).
   \endaligned
  \right.
\end{equation}
Before giving the main results, we need some auxiliary spaces.
We denote the space $S(I)$ with the corresponding norm as follows
\begin{align}\label{Def:SI}
 \|u\|_{S(I)}:=\|\nabla u\|_{L_t^2L_x^6(I\times{\R^3})}+\| u\|_{L_t^{8}L_x^{12}(I\times{\R^3})}.
 \end{align}
In this section, our main result is as follows.
\begin{prop}\label{prop to prove theorem}
Let $0 \in I \subset \R^+$ and suppose that there exists a solution $w\in C(I;\dot{H}_x^1(\R^3))$ of \eqref{PNLS}. Assume that there exists a constant $C_0>0$, such that
\begin{align}\label{7291127}
\|v\|_{S(\R^+)}\leq C_0,
\end{align}
and there exists a constant $E_0>0$, such that
\begin{align}\label{7291056}
{\sup_{t\in I}}\|w(t)\|_{\dot{H}_x^1(\R^3)}\le E_0.
\end{align}
Then, there exists some $C=C(C_0, E_0)>0$ (independent of $I$) such that
\begin{align*}
\|w\|_{S(I)}\le C(C_0,E_0).
\end{align*}
\end{prop}

To prove Proposition \ref{prop to prove theorem}, we shall use the perturbation theory, which shows that the solution $w$ of equation \eqref{PNLS} can stay close to the solution $\widetilde{w}$ of the original energy critical equation:
\begin{equation}\label{4.2}
   \left\{ \aligned
    &i\partial_t \widetilde{w}+\Delta \widetilde{w} = |\widetilde{w}|^4\widetilde{w},
    \\
    &\widetilde{w}(0,x)=\widetilde{w}_0(x),
   \endaligned
  \right.
 \end{equation}
 where $\widetilde{w}=\widetilde{w}(t,x): \R\times\R^3\rightarrow \C$ is a complex-valued function. From the result in \cite{Colliander-Keel-Staffilani-Takaoka-Tao-2008}, we have that if $\wt w_0\in \dot H_x^1(\R^3)$, then the equation \eqref{4.2} is globally well-posed and scatters, and $\widetilde{w}$ satisfies
\begin{align}\label{9111150}
\|\nabla \widetilde{w}\|_{S^0(\R^+)}+\|\widetilde{w}\|_{S(\R^+)}\le C(\|\widetilde{w}_0\|_{\dot{H}_x^1(\R^3)}).
\end{align}

Let $g(t,x):=w(t,x)-\widetilde{w}(t,x)$, then $g$ satisfies the following equation:
 \begin{equation}\label{gpnls}
  \left\{ \aligned
    &i\partial_t g+\Delta g =F(g+v,\widetilde{w}),
    \\
    &g(0,x)=w_0(x)-\widetilde{w}_0(x).
   \endaligned
  \right.
 \end{equation}
Where we denote
\begin{align*}
F(g+v,\widetilde{w})=|g+v+\widetilde{w}|^4(g+v+\widetilde{w})-|\widetilde{w}|^4\widetilde{w}.
\end{align*}
Then our perturbation result regarding to $g$ is as follows.

\begin{lem}\label{long time}
 Let $I\subset \R^+$, $0\in I$ and $E_0>0$. Let $\widetilde{w}\in C(I;\dot{H}_x^1(\R^3))$ be the solution of \eqref{4.2} on $I$ with
 $$
 \widetilde{w}_0=w_0.
 $$
Then, there exists $\eta_1=\eta_1(E_0)$ with the following properties. Assume that
 \begin{align}\label{72611}
 \|v\|_{S(I)}\leq \eta_1,
 \end{align}
and
\begin{align}\label{7261140}
\|w_0\|_{\dot H_x^1(\R^3)}\leq E_0,
\end{align}
then there exists a solution $g\in C(I;\dot{H}_x^1(\R^3))$ of \eqref{PNLS} with initial data $g_0=0$ satisfying
\begin{align*}
\|g\|_{L_t^{\infty}\dot{H}_x^1(I\times\R^3)}+\|g\|_{S(I)}\leq C(E_0, \eta_1)\eta_1.
\end{align*}
\end{lem}

To prove this lemma, we first derive the nonlinear estimates.
 \begin{lem}\label{lemF}
 Let $I\subset \R^+$ be some compact interval and $0\in I$. Then
 \begin{align*}
 \|\nabla F(g+v,\widetilde{w})\|_{N^0(I)}\lesssim& (\|g\|_{S(I)}+\|v\|_{S(I)})(\|g\|_{S(I)}^4+\|v\|_{S(I)}^4+\|\widetilde{w}\|_{S(I)}^4)\\
 &+\|\widetilde{w}\|_{S(I)}(\|g\|_{S(I)}^4+\|v\|_{S(I)}^4).
 \end{align*}
 \begin{proof}
By the definition of $N^0(I)$ and H\"{o}lder's inequality, we have
\begin{align*}
\|\nabla u_1 u_2^4\|_{N^0(I)}&\lesssim \|\nabla u_1 u_2^4\|_{L_t^{1}L_x^{2}(I\times{\R^3})}\\
&\lesssim\|\nabla u_1 \|_{L_t^2L_x^6(I\times{\R^3})}\| u_2\|_{L_t^{8}L_x^{12}(I\times{\R^3})}^4\\
&\lesssim\|u_1 \|_{S(I)}\| u_2\|_{S(I)}^4.
\end{align*}
Moreover, for the term $F(g+v,\widetilde{w})$, we have the pointwise estimates,
  \begin{align*}
  |\nabla F(g+v,\widetilde{w})|&=|\nabla(|g+v+\widetilde{w}|^4(g+v+\widetilde{w})-|\widetilde{w}|^4\widetilde{w})|\\
  &\lesssim (|\nabla g|+|\nabla v|)(|g|^4+|v|^4+|\widetilde{w}|^4)+|\nabla \widetilde{w}|(|g|^4+|v|^4).
  \end{align*}
Hence, this lemma follows by combining the above estimates.
 \end{proof}
 \end{lem}
 Next, we give the proof of the perturbation theory.

\begin{proof}[Proof of Lemma \ref{long time}]
By the assumption of \eqref{7261140}, and \eqref{9111150}, we have
\EQ{
\|\widetilde{w}\|_{S(I)}\leq C(E_0).
}
Fix some absolutely small $0<\eta_2\ll1$, then we can split $I=\cup_{j=1}^J I_j$, $I_j=[t_{j-1}, t_j]$, $t_0=0$, such that
\begin{align}\label{7261123}
\frac{1}{2}\eta_2\leq\|\widetilde{w}\|_{S(I_j)}\leq\eta_2.
\end{align}
Then $J=J(E_0, \eta_2)$ is finite. We also take some $\eta_1\le \eta_2$ that will be decided later. The proof of this lemma will now be accomplished in two steps.

$Step$ 1. In this step, we are going to prove that under the assumption that for any $j\geq 1$,
\begin{align}\label{7261624}
\|g(t_{j-1})\|_{\dot{H}_x^1(\R^3)}\leq \eta_2,
\end{align}
then there exists some suitable constant $B_0>1$ independent of $j$, $\eta_1$, and $\eta_2$, such that
\begin{align}\label{7261659}
\|g\|_{L_t^{\infty}\dot{H}_x^1(I_j\times\R^3)}+\|g\|_{S(I_j)}\leq B_0(\|g(t_{j-1})\|_{\dot{H}_x^1(\R^3)}+\eta_1).
\end{align}

To prove \eqref{7261659}, first recall the Duhamel formula to the equation \eqref{gpnls},
\begin{align}\label{4.74298384}
g(t)=e^{i(t-t_{j-1})\Delta}g(t_{j-1})-i\int_{t_{j-1}}^te^{i(t-s)\Delta} F(g+v,\widetilde{w})(s)ds.
\end{align}
Using Duhamel formula \eqref{4.74298384}, Lemma \ref{Strichartz estimate}, and Sobolev inequality, we have
\begin{align}\label{4.82928191}
\|g\|_{L_t^{\infty}\dot{H}_x^1(I_j\times\R^3)}+\|g\|_{S(I_j)}
\lesssim& \|g(t_{j-1})\|_{\dot{H}_x^1(\R^3)}+\Big\|\int_{t_{j-1}}^te^{i(t-s)\Delta} \nabla F(g+v,\widetilde{w})(s)ds\Big\|_{S^0(I_j)}\nonumber\\
\lesssim& \|g(t_{j-1})\|_{\dot{H}_x^1(\R^3)}+\big\|\nabla F(g+v,\widetilde{w})\big\|_{N^0(I_j)}.
\end{align}
Moreover, by the Lemma \ref{lemF}, \eqref{72611}, and \eqref{7261123}, we have
\begin{align*}
\big\|\nabla F(g+v,\widetilde{w})\big\|_{N^0(I_j)}\lesssim&(\|g\|_{S(I_j)}+\eta_1)(\|g\|_{S(I_j)}^4+\eta_1^4+\eta_2^4)+\eta_2(\|g\|_{S(I_j)}^4+\eta_1^4)\\
 \lesssim&(\eta_2^4 + \|g\|_{S(I_j)}^4 )\|g\|_{S(I_j)}+\eta_1\eta_2^4.
\end{align*}
Further by \eqref{4.82928191} implies that
\begin{align}\label{4.6}
\|g\|_{L_t^{\infty}\dot{H}_x^1(I_j\times\R^3)}+\|g\|_{S(I_j)}\lesssim&\|g(t_{j-1})\|_{\dot{H}_x^1(\R^3)}+ (\eta_1^4 + \|g\|_{S(I_j)}^4 )\|g\|_{S(I_j)}+\eta_1.
\end{align}
Noting the assumptions of smallness conditions \eqref{7261624}, by \eqref{4.6} and the bootstrap method, we can obtain
\begin{align}\label{4.9}
\|g\|_{L_t^{\infty}\dot{H}_x^1(I_j\times\R^3)}+\|g\|_{S(I_j)}\lesssim \|g(t_{j-1})\|_{\dot{H}_x^1(\R^3)}+\eta_1.
\end{align}
Hence, we can choose suitable constant $B_0>1$, such that \eqref{7261659} holds.

$Step$ 2.
 Next, we shall to get the desired results through induction. To start with, we take the parameter $\eta_1$ such that
\EQ{\label{defn:eta_1}
JB_0^J\eta_1 \leq\eta_2.
}
Therefore, $\eta_1$ depends only on $E_0$ and the absolute small constant $\eta_2$.

First, we consider the subinterval $I_1=[0, t_1]$. In this case, $g_0=w_0-\widetilde{w}_0=0$.
Thus, we have
$$
\|g_0\|_{\dot{H}^1(\R^3)}=0.
$$
By the first step, we get the existence of $g$ on $I_1$, and
\begin{align}\label{4.666}
\|g\|_{L_t^{\infty}\dot{H}_x^1(I_1\times\R^3)}+\|g\|_{S(I_1)}\leq B_0\eta_1.
\end{align}

Secondly, we consider the subinterval $I_2=[t_1, t_2]$. In this case, by \eqref{4.666} and \eqref{defn:eta_1},
\begin{align*}
\|g(t_1)\|_{\dot{H}_x^1(\R^3)}\leq B_0\eta_1\leq\eta_2.
\end{align*}
The above estimate satisfies the assumption \eqref{7261624}.
Thus, by $Step$ 1, we have
\begin{align}\label{4.0000}
\|g\|_{L_t^{\infty}\dot{H}_x^1(I_2\times\R^3)}+\|g\|_{S(I_2)}&\leq B_0( B_0\eta_1+\eta_1)\leq 2B_0^2\eta_1.
\end{align}

Now we start the induction procedure from $I_2$. We aim to prove that for any $j=1,2,\ldots, J$,
\begin{align}\label{4.78000}
\|g\|_{L_t^{\infty}\dot{H}_x^1(I_j\times\R^3)}+\|g\|_{S(I_j)}\leq jB_0^j\eta_1.
\end{align}
For $j=2$ case, we have that the above estimate holds. Next, for $j=k$ case, we suppose the above estimate \eqref{4.78000} holds, that is
\begin{align}\label{4.123456}
\|g\|_{L_t^{\infty}\dot{H}_x^1(I_k\times\R^3)}+\|g\|_{S(I_k)}\leq kB_0^k\eta_1.
\end{align}
Hence, it suffices to prove the $j=k+1$ case. Using \eqref{4.123456} and \eqref{defn:eta_1}, we obtain
\begin{align}\label{4.909090}
\|g(t_k)\|_{\dot{H}_x^1(\R^3)}\le  kB_0^k\eta_1 \le JB_0^J\eta_1 \leq\eta_2.
\end{align}
By $Step$ 1, we obtain the existence of $g$ on $I_{k+1}$, and
\begin{align}\label{4.290000}
\|g\|_{L_t^{\infty}\dot{H}_x^1(I_{k+1}\times\R^3)}+\|g\|_{S(I_{k+1})}\leq B_0(kB_0^k\eta_1+\eta_1)\leq(k+1)B_0^{k+1}\eta_1.
\end{align}
Hence, by induction, we have \eqref{4.78000} holds for any $j=1, 2,\ldots, J$.

Then, we have the existence of $g$ on the whole interval $I$, and for any $j$,\begin{align*}
\|g\|_{L_t^{\infty}\dot{H}_x^1(I_j\times\R^3)}+\|g\|_{S(I_j)}\leq jB_0^j\eta_1\leq\eta_2.
\end{align*}
Summing this over all subintervals $I_j$, we complete the proof of this lemma.
\end{proof}

Finally, we are in a position to prove Proposition \ref{prop to prove theorem}. Although the perturbation theory is a local result, we can apply it by iteration to proving that the global spacetime norms of $w$ is uniformly bounded.
\begin{proof}[Proof of Proposition \ref{prop to prove theorem}]
For any fixed $t_0\in I$, we consider the following equation
\begin{equation}\label{5.111222}
   \left\{ \aligned
    &i\partial_t \widetilde{w}+\Delta \widetilde{w} = |\widetilde{w}|^4\widetilde{w},
    \\
    &\widetilde{w}(t_0 ,x)=w(t_0, x).
   \endaligned
  \right.
 \end{equation}
By \eqref{9111150}, we have that there exists a global solution $\widetilde{w}(t,x)=\widetilde{w}^{(t_0)}(t,x)$ of the equation \eqref{5.111222} with
 \begin{align}\label{5.2}
\|\widetilde{w}^{(t_0)}\|_{S(\R^{+})}\le C(\|w(t_0)\|_{\dot{H}_x^1(\R^3)}).
 \end{align}

Now, let $\eta_1=\eta_1(E_0)$ be defined as in Lemma \ref{long time}. By the assumption \eqref{7291127},  we can split $I=\cup_{l=1}^L \widetilde{I}_l$, $\widetilde{I}_l=[\tau_{l-1}, \tau_l]$, $\tau_0=0$, such that
\begin{align*}
\frac{1}{2}\eta_1\leq\|v\|_{S(\widetilde{I}_l)}\leq \eta_1.
\end{align*}
 Then $L(C_0, \eta_1)$ is finite.
We consider subinterval $\widetilde{I}_1$ firstly. By \eqref{7291056}, we can take $t_0=\tau_0=0$ for the equation \eqref{5.111222}, and clearly have
\begin{align*}
\|w_0\|_{\dot H_x^1(\R^3)}\leq E_0.
\end{align*}
Then, using Lemma \ref{long time} on $\widetilde{I}_1$, we obtain the existence of $w\in C\big(\widetilde{I}_1;\dot{H}_x^1(\R^3)\big)$.

Similarly as above, for $\widetilde{I}_2$, we can take $t_0=\tau_1$ for the equation \eqref{5.111222}. Using \eqref{7291056} again,
\begin{align*}
{\sup_{t\in[0, \tau_1]}}\|w(t)\|_{\dot{H}_x^1(\R^3)}\le E_0.
\end{align*}
Particularly,
\begin{align*}
\|\wt w(\tau_1)\|_{\dot{H}_x^1(\R^3)}=\|w(\tau_1)\|_{\dot{H}_x^1(\R^3)}\le E_0.
\end{align*}
Then, we can apply Lemma \ref{long time} on $\widetilde{I}_2$ after translation in $t$ from the starting point $\tau_1$, and obtained the well-posedness on $\widetilde{I}_2$.

Inductively, under the assumptions of Proposition \ref{prop to prove theorem}, we can obtain the existence of $w\in C\big(\widetilde{I}_l;\dot{H}_x^1(\R^3)\big)$ for $l=1,2,\ldots, L$. Moreover, from Lemma \ref{long time} and \eqref{9111150}, we have
\begin{align}\label{8101864}
\|w\|_{S(I)}\le C(C_0, E_0).
\end{align}
Hence, we finish the proof of the proposition.
\end{proof}

\section{Framework of the proof}
\subsection{Linear and nonlinear decomposition}
Now, we turn to the proof of Theorem \ref{main theorem}. We consider the energy critical NLS \eqref{NLS} with the initial data $u_0=f_+$.
Fixing $\delta_0>0$, we take $N=N(\delta_0)>0$, such that
\begin{align}\label{1.12}
\|P_{\geq N} \chi_{\geq 1}f\|_{H^{s_0}(\R^3)}\leq \delta_0.
\end{align}
Denote $v_0:=(P_{\geq N} \chi_{\geq 1}f)_{out}$, and $w_0:=\frac{1}{2}P_{\leq 1}f+\frac{1}{2}P_{\geq 1} \chi_{\leq 1}f+(P_{1\leq\cdot\leq N} \chi_{\geq 1}f)_{out}$. Then we split the solution $u$ of \eqref{NLS} as $u=v+w$,
where
\begin{align*}
v=e^{it\Delta}v_0,
\end{align*}
and $w$ satisfies the equation
 \begin{equation}\label{PNLS-2}
   \left\{ \aligned
    &i\partial_t w+\Delta w =|w+v|^4(w+v),
    \\
    &w(0,x)=w_0(x).
   \endaligned
  \right.
\end{equation}

Now we need the following two hypotheses: assume that there exist constants $C_0, E_0>0$,
(H1)
 \begin{align*}
 \|v\|_{S(\R^+)}\leq C_0,
\end{align*}
(H2)
For any $0 \in I \subset \R^+$, if $w\in C(I;\dot{H}_x^1(\R^3))$, then
\begin{align*}
{\sup_{t\in I}}\|w(t)\|_{\dot{H}_x^1(\R^3)}\le E_0.
\end{align*}

Then, under the above two hypotheses, the general theory in the third section is available for our case, we can obtain that the solution $w$ to the above perturbation equation is global in $\dot{H}^1$.
Hence, it suffices to verify the above two hypotheses (H1) and (H2).
\subsection{Proof of Theorem \ref{main theorem} under the hypotheses (H1) and (H2)}
We are now in a position to give the proof of Theorem \ref{main theorem}.
\begin{proof}
First of all, by Proposition \ref{prop to prove theorem}, under the hypotheses (H1) and (H2), we obtain the global existence of $w$ in the forward time and
\begin{align}\label{5.5}
\|w\|_{S(\R^+)}\leq  C(C_0,E_0).
\end{align}

Next, we prove the scattering statement in Theorem \ref{main theorem}.
Set
\begin{align*}
u_+=f_+-i\int_0^{+\infty}e^{-is\Delta}(|u|^4u)(s)ds.
\end{align*}
Then we have
\begin{align*}
u(t)-e^{it\Delta}u_+=i\int_t^{+\infty}e^{i(t-s)\Delta}(|u|^4u)(s)ds.
\end{align*}
Now, we claim that $u_+-f_+ \in  H^{1}$ and
\begin{align*}
\|u(t)-e^{it\Delta}u_+\|_{H_x^1(\R^3)}\rightarrow0,\quad as \quad t \rightarrow +\infty.
\end{align*}
Indeed, it reduced to prove that
\begin{align}\label{1229002}
\Big\|\int_0^{+\infty}e^{i(t-s)\Delta}(|u|^4u)(s)ds\Big\|_{{H}_x^1(\R^3)}< +\infty.
\end{align}
First, by Strichartz's estimate \eqref{1.222234} and H\"{o}lder's inequality, we have
\begin{align}\label{5.6}
\Big\|\int_0^{+\infty}e^{i(t-s)\Delta}(|u|^4u)(s)ds\Big\|_{\dot{H}_x^1(\R^3)}&\lesssim\|\nabla(|u|^4u)\|_{L_t^1L_x^{2}(\R^+\times\R^3)}\nonumber\\
&\lesssim \|\nabla u\|_{L_t^{2}L_x^{6}(\R^+\times\R^3)}\|u\|_{L_t^{8}L_x^{12}(\R^+\times\R^3)}^4.
\end{align}
Moreover, by (H1) and \eqref{5.5}, we have
\begin{align}\label{1229001}
\|\nabla u\|_{L_t^{2}L_x^{6}(\R^+\times\R^3)}+\|u\|_{L_t^{8}L_x^{12}(\R^+\times\R^3)}&\lesssim\|\nabla w\|_{L_t^{2}L_x^{6}(\R^+\times\R^3)}+\|w\|_{L_t^{8}L_x^{12}(\R^+\times\R^3)}+\|v\|_{S(\R^+)}\nonumber\\
&\le C(C_0,E_0).
\end{align}
Combining \eqref{5.6} with \eqref{1229001}, we get
\begin{align*}
\Big\|\int_0^{+\infty}e^{i(t-s)\Delta}(|u|^4u)(s)ds\Big\|_{\dot{H}_x^1(\R^3)}< +\infty.
\end{align*}
Next, similarly as above, we have
\begin{align*}
\Big\|\int_0^{+\infty}e^{i(t-s)\Delta}(|u|^4u)(s)ds\Big\|_{L_x^2(\R^3)}&\lesssim \big\||u|^4u\big\|_{L_t^2L_x^{\frac{6}{5}}(\R^+\times\R^3)}\\
&\lesssim\|u\|_{L_t^{8}L_x^{12}(\R^+\times\R^3)}^4\|u\|_{L_t^{\infty}L_x^{2}(\R^+\times\R^3)}\\
&< +\infty.
\end{align*}
Hence, we obtain \eqref{1229002}.
This proves the scattering statement and thus finishes the proof of Theorem \ref{main theorem}.
\end{proof}

\section{Linear estimates}
In this section, we give the proof of the validity of hypothesis (H1).
\begin{prop}\label{linear prop}
Let $s_0\in (\frac{5}{6}, 1)$. Then there exists a constant $C_0>0$, such that
\begin{align*}
\|v\|_{S(\R^+)}\leq C_0.
\end{align*}
\end{prop}
First of all, we need the frequency restricted incoming component of $f$ as follows, for any fixed integer $k$
\begin{align*}
f_{in, k}(r)=r^{-\beta}\int_0^{+\infty}\big(J(-\rho r)+K(\rho r)\big)\chi_{2^k}(\rho)\rho^{-\alpha+2}\mathcal{F}f(\rho)d\rho;
\end{align*}
correspondingly, the frequency restricted outgoing component of $f$ as follows
\begin{align*}
f_{out, k}(r)=r^{-\beta}\int_0^{+\infty}\big(J(-\rho r)-K(\rho r)\big)\chi_{2^k}(\rho)\rho^{-\alpha+2}\mathcal{F}f(\rho)d\rho.
\end{align*}
\subsection{Known results from \cite{BECEANU-DENG-SOFFER-WU-2021}}
Next, we recall some useful lemmas, see Proposition 3.11, Lemma 3.12, Proposition 4.1 and Proposition 4.3 in \cite{BECEANU-DENG-SOFFER-WU-2021} for the proof.
\begin{lem}\label{lem 0}
Suppose that $f \in L^2(\R^3)$, then
\begin{align*}
\|f_{out/in}\|_{L^2(\R^3)}\lesssim \|f\|_{L^2(\R^3)}.
\end{align*}
\end{lem}
The next lemma shows that if a function $f$ has high frequency $f=P_{2^k}f$, then its incoming/outgoing component will have almost the same frequency plus a smooth perturbation.
\begin{lem}\label{lemma 1}
Let $k\geq 0$ be an integer. Suppose that $f\in L^2(\R^3)$ with $suppf\subset \{x: |x|\geq1\}$, then
\begin{align*}
(P_{2^k}(\chi_{\geq1}f))_{out/in}=(P_{2^k}(\chi_{\geq1}f))_{out/in,k-1\leq\cdot\leq k+1}+h_k,
\end{align*}
where $h_k$ satisfies the following estimate,
\begin{align*}
\|h_k\|_{H^2(\R^3)}\lesssim 2^{-10k}\|P_{2^k}(\chi_{\geq1}f)\|_{L^2(\R^3)}.
\end{align*}
Moreover,
\begin{align*}
\|\chi_{\leq\frac{1}{4}}(P_{2^k}(\chi_{\geq1}f))_{out/in,k-1\leq\cdot\leq k+1}\|_{H^2(\R^3)}\lesssim 2^{-2k}\|P_{2^k}(\chi_{\geq1}f)\|_{L^2(\R^3)}.
\end{align*}
\end{lem}
The following result is the incoming/outgoing linear flow's estimate related to the inside region.
\begin{lem}\label{lemma 2}
Let $k\geq 0$ be an integer. Then there exists $\delta>0$, such that for any triple $(\gamma, q, r)$ satisfying that
\begin{align*}
q\geq2, \quad r>2, \quad 0\leq \gamma\leq1, \quad \frac{2}{q}+\frac{5}{r}<\frac{5}{2},
\end{align*}
the following estimates holds,
\begin{align*}
\big\||&\nabla|^{\gamma}[\chi_{\leq \delta(1+2^kt)}e^{it\Delta}(\chi_{\geq \frac{1}{4}}(P_{2^k}(\chi_{\geq1}f))_{out,k-1\leq\cdot\leq k+1})]\big\|_{L_t^qL_x^r(\R^+\times\R^3)}\\
&\lesssim 2^{-(2-(\gamma-\frac{2}{q}-\frac{3}{r}))k}\|P_{2^k}(\chi_{\geq1}f)\|_{L^2(\R^3)}.
\end{align*}
The same estimate holds when $e^{it\Delta}$ and $_{out}$ are replaced by $e^{-it\Delta}$ and $_{in}$, respectively.
\end{lem}
The following result is the incoming/outgoing linear flow's estimate related to the outside region.
\begin{lem}\label{lemma 3}
Let $k\geq 0$ be an integer. Moreover, let $r$, $\gamma_1$, $\gamma_2$, $s$ be the parameters satisfying
\begin{align*}
r>2, \quad \gamma_1>2, \quad \gamma_2\geq0, \quad s+\frac{1}{r}\geq \frac{1}{2}, \quad \gamma_1+s=\frac{3}{2}-\frac{3}{r}.
\end{align*}
Then for any $t>0$,
\begin{align*}
\big\||&\nabla|^{\gamma_2}[\chi_{\geq \delta(1+2^kt)}e^{it\Delta}(\chi_{\geq \frac{1}{4}}(P_{2^k}(\chi_{\geq1}f))_{out,k-1\leq\cdot\leq k+1})]\big\|_{L_x^r(\R^3)}\\
&\lesssim (1+2^kt)^{-\gamma_1}2^{(\gamma_2+s+)k}\|P_{2^k}(\chi_{\geq1}f)\|_{L^2(\R^3)}.
\end{align*}
The same estimate holds when $e^{it\Delta}$ and $_{out}$ are replaced by $e^{-it\Delta}$ and $_{in}$, respectively.
\end{lem}
\subsection{Further estimates}
Now, we can obtain the following space-time estimates based on the above lemma.
\begin{cor}\label{Corollary-4}
Let $(\gamma, q, r)$ be the triple satisfying
\begin{align*}
\gamma\geq 0, \quad q\geq 1, \quad r>2, \quad \frac{1}{q}<1-\frac{2}{r},
\end{align*}
then the following estimate holds,
\begin{align}\label{830851}
\big\||&\nabla|^{\gamma}[\chi_{\geq \delta(1+2^kt)}e^{it\Delta}(\chi_{\geq \frac{1}{4}}(P_{2^k}(\chi_{\geq1}f))_{out,k-1\leq\cdot\leq k+1})]\big\|_{L_t^qL_x^r(\R^+\times\R^3)}\nonumber\\
&\lesssim 2^{(-\frac{1}{q}-\frac{1}{r}+\frac{1}{2}+\gamma+)k}\|P_{2^k}(\chi_{\geq1}f)\|_{L^2(\R^3)}.
\end{align}
Moreover, for any $\delta>0$,
\begin{align}\label{830853}
\big\||&\nabla|^{\gamma}[\chi_{\geq \delta(1+2^kt)}e^{it\Delta}(\chi_{\geq \frac{1}{4}}(P_{2^k}(\chi_{\geq1}f))_{out,k-1\leq\cdot\leq k+1})]\big\|_{L_t^2L_x^{\infty}([\delta, +\infty)\times\R^3)}\nonumber\\
&\lesssim_\delta 2^{(-\frac{1}{2}+\gamma+)k}\|P_{2^k}(\chi_{\geq1}f)\|_{L^2(\R^3)}.
\end{align}
The same estimate holds when $e^{it\Delta}$ and $_{out}$ are replaced by $e^{-it\Delta}$ and $_{in}$, respectively.
\begin{proof}
\eqref{830851} was proved by Corollary 4.4 in \cite{BECEANU-DENG-SOFFER-WU-2021}. In addition, we also need the result \eqref{830853}.

From Lemma \ref{lemma 3}, we have
\begin{align*}
\big\||\nabla|^{\gamma}&[\chi_{\geq \delta(1+2^kt)}e^{it\Delta}(\chi_{\geq \frac{1}{4}}(P_{2^k}(\chi_{\geq1}f))_{out,k-1\leq\cdot\leq k+1})]\big\|_{L_t^2L_x^{\infty}([\delta, +\infty)\times\R^3)}\nonumber\\
&\lesssim 2^{(\gamma+\frac{1}{2}+)k}\|(1+2^kt)^{-1}\|_{L_t^2([\delta, +\infty))}\|P_{2^k}(\chi_{\geq1}f)\|_{L^2(\R^3)}.
\end{align*}
Furthermore, we have
\begin{align*}
\|(1+2^kt)^{-1}\|_{L_t^2([\delta, +\infty))}\lesssim_\delta 2^{-k}.
\end{align*}
Combining the above two estimates, we prove the corollary.
\end{proof}
\end{cor}
Next, we gather some space-time norms that will be used below. Define the $Y(I)$ space by its norm
 \begin{align*}
 \|v\|_{Y(I)}:=&N^{s_0-\frac56-}\|\nabla v\|_{L_t^2L_x^6(I\times\R^3)}+N^{s_0-\frac 7 {24}-}\|v\|_{L_t^{8}L_x^{12}(I\times\R^3)}\\
 &+N^{s_0-\frac13-}\|v\|_{L_t^{\infty}L_x^6(I\times\R^3)}+N^{s_0-}\|v\|_{L_t^2L_x^{\infty}(I\times\R^3)}.
 \end{align*}
Then, by the above Lemmas and Corollary, we have
\begin{lem}\label{linear estimate}
Let $N\geq 1$, $\frac{5}{6}<s_0<1$, then the following estimates hold,
\begin{align*}
\|v\|_{Y(\R^+)}\lesssim\|P_{\geq N}\chi_{\geq 1}f\|_{H^{s_0}(\R^3)}.
\end{align*}
Moreover, for any $\delta>0$,
\begin{align*}
\|\nabla v\|_{L_t^2L_x^{\infty}([\delta, +\infty)\times \R^3)}\lesssim N^{-s_0+\frac12+}\|P_{\geq N}\chi_{\geq1}f\|_{H^{s_0}(\R^3)}.
\end{align*}
\begin{proof}
The estimates above $v$ on $\R^+$ was proved by Proposition 4.5 in \cite{BECEANU-DENG-SOFFER-WU-2021}, we only sketch the proof for completeness, and prove the spacetime estimate of $v$ on $[\delta, +\infty)$.
Let $N=2^{k_0}$ for some $k_0 \in \N$. By Lemma \ref{lemma 1}, we write
\begin{align*}
 v=&e^{it\Delta}(P_{\geq N}\chi_{\geq 1}f)_{out}\\=&\sum_{k=k_0}^{\infty}e^{it\Delta}(P_{2^k}(\chi_{\geq 1}f))_{out}\\
 =&\sum_{k=k_0}^{\infty}e^{it\Delta}(\chi_{\leq \frac{1}{4}}(P_{2^k}(\chi_{\geq 1}f))_{out, k-1\leq\cdot\leq k+1})+\sum_{k=k_0}^{\infty}e^{it\Delta}h_k\\
 &+\sum_{k=k_0}^{\infty}e^{it\Delta}(\chi_{\geq \frac{1}{4}}(P_{2^k}(\chi_{\geq 1}f))_{out, k-1\leq\cdot\leq k+1})\\
 =&\sum_{k=k_0}^{\infty}e^{it\Delta}(\chi_{\leq \frac{1}{4}}(P_{2^k}(\chi_{\geq 1}f))_{out, k-1\leq\cdot\leq k+1})+\sum_{k=k_0}^{\infty}e^{it\Delta}h_k\\
  &+\sum_{k=k_0}^{\infty}\chi_{\leq \delta(1+2^kt)}e^{it\Delta}(\chi_{\geq \frac{1}{4}}(P_{2^k}(\chi_{\geq 1}f))_{out, k-1\leq\cdot\leq k+1})\\
  &+\sum_{k=k_0}^{\infty}\chi_{\geq \delta(1+2^kt)}e^{it\Delta}(\chi_{\geq \frac{1}{4}}(P_{2^k}(\chi_{\geq 1}f))_{out, k-1\leq\cdot\leq k+1}).
 \end{align*}
Then by Lemma \ref{Strichartz estimate} and Lemma \ref{lemma 1}, we obtain
\begin{align}\label{12201}
\sum_{k=k_0}^{\infty}&\big\|e^{it\Delta}|\nabla|(\chi_{\leq \frac{1}{4}}(P_{2^k}(\chi_{\geq 1}f))_{out, k-1\leq\cdot\leq k+1})\big\|_{L_t^2L_x^6(\R^+\times\R^3)}\nonumber\\
\lesssim&\sum_{k=k_0}^{\infty}\|\chi_{\leq\frac{1}{4}}(P_{2^k}(\chi_{\geq1}f))_{out,k-1\leq\cdot\leq k+1}\|_{H^2(\R^3)}\nonumber\\
\lesssim& \sum_{k=k_0}^{\infty}2^{-2k}\|P_{2^k}(\chi_{\geq1}f)\|_{L^2(\R^3)}\nonumber\\
\lesssim& N^{-2-s_0}\|P_{\geq N}\chi_{\geq1}f\|_{H^{s_0}(\R^3)}.
\end{align}
In the same way as above, we can also obtain
\begin{align}\label{12202}
\sum_{k=k_0}^{\infty}\|e^{it\Delta}|\nabla|h_k\|_{L_t^2L_x^6(\R^+\times\R^3)}\lesssim N^{-2-s_0}\|P_{\geq N}\chi_{\geq1}f\|_{H^{s_0}(\R^3)}.
\end{align}
Next, by Lemma \ref{lemma 2}, noting that $\frac56<s_0<1$, we have
\begin{align}\label{12203}
\sum_{k=k_0}^{\infty}&\big\||\nabla|\big[\chi_{\leq \delta(1+2^kt)}e^{it\Delta}(\chi_{\geq \frac{1}{4}}(P_{2^k}(\chi_{\geq 1}f))_{out, k-1\leq\cdot\leq k+1})\big]\big\|_{L_t^2L_x^6(\R^+\times\R^3)}\nonumber\\
\lesssim& \sum_{k=k_0}^{\infty}2^{-(2-(1-\frac{2}{2}-\frac{3}{6}))k}\|P_{2^k}(\chi_{\geq1}f)\|_{L^2(\R^3)}\nonumber\\
\lesssim& N^{-s_0-\frac{5}{2}}\|P_{\geq N}\chi_{\geq 1}f\|_{H^{s_0}(\R^3)}.
\end{align}
Finally, from Corollary \ref{Corollary-4}, noting that $\frac56<s_0<1$, we obtain that
\begin{align}\label{12204}
\sum_{k=k_0}^{\infty}&\big\||\nabla|\big[\chi_{\geq \delta(1+2^kt)}e^{it\Delta}(\chi_{\geq \frac{1}{4}}(P_{2^k}(\chi_{\geq 1}f))_{out, k-1\leq\cdot\leq k+1})\big]\big\|_{L_t^2L_x^6(\R^+\times\R^3)}\nonumber\\
\lesssim& \sum_{k=k_0}^{\infty}2^{(-\frac{1}{2}-\frac16+\frac12+1+)k}\|P_{2^k}(\chi_{\geq1}f)\|_{L^2(\R^3)}\nonumber\\
\lesssim& N^{-(s_0-\frac56)+}\|P_{\geq N}\chi_{\geq 1}f\|_{H^{s_0}(\R^3)}.
\end{align}
Then collecting the estimates \eqref{12201}-\eqref{12204}, we get
\begin{align*}
\|\nabla v\|_{L_t^2L_x^6(\R^+\times\R^3)}\lesssim N^{-s_0+\frac56+}\|P_{\geq N}\chi_{\geq 1}f\|_{H^{s_0}(\R^3)}.
\end{align*}
Similarly as above, by Lemma \ref{Strichartz estimate}, Lemma \ref{lemma 1}, Lemma \ref{lemma 2}, Corollary \ref{Corollary-4}, and Sobolev inequality, we can obtain
\begin{align*}
\|v\|_{L_t^{8}L_x^{12}(\R^+\times\R^3)}\lesssim N^{-s_0+\frac 7 {24}+}\|P_{\geq N}\chi_{\geq 1}f\|_{H^{s_0}(\R^3)},
\end{align*}
\begin{align*}
\|v\|_{L_t^{\infty}L_x^6(\R^+\times\R^3)}\lesssim N^{-s_0+\frac13+}\|P_{\geq N}\chi_{\geq 1}f\|_{H^{s_0}(\R^3)},
\end{align*}
and
\begin{align*}
\|v\|_{L_t^2L_x^{\infty}(\R^+\times\R^3)}\lesssim N^{-s_0+}\|P_{\geq N}\chi_{\geq 1}f\|_{H^{s_0}(\R^3)}.
\end{align*}

Next, we will estimate the term $\|\nabla v\|_{L_t^2L_x^{\infty}([\delta, +\infty)\times \R^3)}$, the proof is similar as above.
By Lemma \ref{Strichartz estimate}, Lemma \ref{lemma 1} and Lemma \ref{lemma 2}, we have
\begin{align*}
\sum_{k=k_0}^{\infty}&\big\|e^{it\Delta}|\nabla|(\chi_{\leq \frac{1}{4}}(P_{2^k}(\chi_{\geq 1}f))_{out, k-1\leq\cdot\leq k+1})\big\|_{L_t^2L_x^{\infty}([\delta, +\infty)\times \R^3)}\\
+&\sum_{k=k_0}^{\infty}\big\||\nabla|\big[\chi_{\leq \delta(1+2^kt)}e^{it\Delta}(\chi_{\geq \frac{1}{4}}(P_{2^k}(\chi_{\geq 1}f))_{out, k-1\leq\cdot\leq k+1})\big]\big\|_{L_t^2L_x^{\infty}([\delta, +\infty)\times \R^3)}\\
+&\sum_{k=k_0}^{\infty}\|e^{it\Delta}|\nabla|h_k\|_{L_t^2L_x^{\infty}([\delta, +\infty)\times \R^3)}\\
\lesssim& N^{-2-s_0}\|P_{\geq N}\chi_{\geq1}f\|_{H^{s_0}(\R^3)}.
\end{align*}
Furthermore, by Corollary \ref{Corollary-4}, we obtain
\begin{align*}
\sum_{k=k_0}^{\infty}&\big\||\nabla|\big[\chi_{\geq \delta(1+2^kt)}e^{it\Delta}(\chi_{\geq \frac{1}{4}}(P_{2^k}(\chi_{\geq 1}f))_{out, k-1\leq\cdot\leq k+1})\big]\big\|_{L_t^2L_x^{\infty}([\delta, +\infty)\times \R^3)}\\
\lesssim& N^{-s_0+\frac12+}\|P_{\geq N}\chi_{\geq1}f\|_{H^{s_0}(\R^3)}.
\end{align*}
Hence, by the above two estimates, we obtain
\begin{align*}
\|\nabla v\|_{L_t^2L_x^{\infty}([\delta, +\infty)\times \R^3)}\lesssim N^{-s_0+\frac12+}\|P_{\geq N}\chi_{\geq1}f\|_{H^{s_0}(\R^3)}.
\end{align*}
Thus, we get the desired estimates and complete the proof of the lemma.
\end{proof}
\end{lem}
Hence, by Lemma \ref{linear estimate}, we can obtain Proposition \ref{linear prop} by choosing a suitable constant $C_0>0$.

\section{A priori estimate}

In this section, we give the proof of the validity of a priori assumptions (H2). To this end, we define the working space $X_N(I)$ for $I \subset \R^+$ by its norm
\begin{align*}
\|h\|_{X_N(I)}=N^{3(s_0-1)}\|h\|_{L_t^{\infty}\dot{H}_x^1(I\times{\R^3})}+N^{\frac{9}{8}(s_0-1)}\|h\|_{L_{t,x}^{8}(I\times{\R^3})}.
\end{align*}
Then we have
\begin{align}\label{1.2323232}
\|h\|_{L_t^{\infty}\dot{H}_x^1(I\times{\R^3})}\leq N^{3(1-s_0)}\|h\|_{X_N(I)},\nonumber\\
\|h\|_{L_{t,x}^{8}(I\times{\R^3})}\leq N^{\frac{9}{8}(1-s_0)}\|h\|_{X_N(I)}.
\end{align}
The following is the main result in this section.
\begin{prop}\label{priori prop}
Let $s_0\in (\frac{5}{6}, 1)$. Let $w\in C(I;\dot{H}_x^1(\R^3))$ be the solution of the equation \eqref{PNLS}. Then there exists a constant $E_0= E_0(N)>0$, such that
\begin{align*}
{\sup_{t\in I}}\|w(t)\|_{\dot{H}_x^1(\R^3)}\le E_0.
\end{align*}
\end{prop}
To prove Proposition \ref{priori prop}, we need some spacetime norms of $w$ are uniformly bounded.  First of all, we have the initial data $w_0$ is in $\dot{H}^1$. Indeed, by Bernstein's inequality, we have
 \begin{align*}
 \|P_{\leq 1}f+P_{\geq 1} \chi_{\leq 1}f\|_{\dot{H}^1(\R^3)}\lesssim \|\chi_{\leq 1}f\|_{H^1(\R^3)}+\|\chi_{\geq 1}f\|_{H^{s_0}(\R^3)}.
\end{align*}
Moreover, by Lemmas \ref{lem  0} and \ref{lemma 1}, we have
 \begin{align*}
\|(P_{1\leq\cdot\leq N} \chi_{\geq 1}f)_{out}\|_{\dot{H}^1(\R^3)}&\lesssim \sum_{k=0}^{k_0} \big(\|(P_{2^k}(\chi_{\geq 1}f))_{out, k-1\leq\cdot\leq k+1})\|_{\dot{H}^1(\R^3)}+\|h_k\|_{\dot{H}^1(\R^3)}\big)\\
&\lesssim \sum_{k=0}^{k_0}  \big(2^k\|P_{2^k}(\chi_{\geq 1}f)\|_{L^2(\R^3)}+2^{-10k}\|P_{2^k}(\chi_{\geq 1}f)\|_{L^2(\R^3)} \big)\\
&\lesssim N^{1-s_0}\|\chi_{\geq 1}f\|_{H^{s_0}(\R^3)}+\|\chi_{\geq 1}f\|_{H^{s_0}(\R^3)}\\
&\lesssim N^{1-s_0}.
\end{align*}
Hence, combining the above two estimates, we have
\begin{align}\label{10171753}
\|w_0\|_{\dot{H}^1(\R^3)}\lesssim N^{1-s_0}.
\end{align}
Next, we have that the $L_t^{\infty}L_x^2(I \times \R^3)$ norm of $w$ is uniformly bounded. In fact, by the conservation of mass \eqref{1.2} and Lemma \ref{lem 0} , we obtain
\begin{align}\label{1.88}
\|w\|_{L_t^{\infty}L_x^{2}(I\times{\R^3})}&\lesssim \|u\|_{L_t^{\infty}L_x^{2}(I\times{\R^3})}+\|v\|_{L_t^{\infty}L_x^{2}(I\times{\R^3})}\nonumber\\
&\lesssim \|f_+\|_{L^2(\R^3)}+\|f\|_{H^{s_0}(\R^3)}\nonumber\\
&\lesssim \|f\|_{H^{s_0}(\R^3)}.
\end{align}
Now, we start with the Morawetz estimates.
\subsection{Morawetz estimates}
In this subsection, we consider the Morawetz-type estimate by Lin-Strauss \cite{Lin-Strauss-1978}.
\begin{lem}\label{Morawetz}
Let $\frac{5}{6}<s_0<1$ and $\|w\|_{X_N(I)}\geq 1$, then
\begin{align*}
\int_{I}\int_{\R^3}\frac{|w(t,x)|^6}{|x|}dxdt\lesssim N^{3(1-s_0)}(\|w\|_{X_N(I)}+\delta_0\|w\|_{X_N(I)}^5).
\end{align*}
\begin{proof}
Let
\begin{align*}
M(t)=\textrm{Im}\int_{\R^3}\frac{x}{|x|}\cdot \nabla w(t,x)\overline{w}(t,x)dx.
\end{align*}
By integration-by-parts, we have
\begin{align*}
M'(t)&=\textrm{Im}\int_{\R^3}\frac{x}{|x|}\cdot (\nabla w_t\overline{w}+\nabla w\overline{w_t})dx\\
&=(-1)\textrm{Im}\int_{\R^3}(\frac{2}{|x|}w_t\overline{w}+\frac{x}{|x|}\cdot\nabla\overline{w}w_t)dx+\textrm{Im}\int_{\R^3}\frac{x}{|x|}\cdot\nabla w\overline{w_t}dx\\
&=2\textrm{Im}\int_{\R^3}(\frac{x}{|x|}\cdot\nabla w+\frac{1}{|x|}w)\overline{w_t}dx.
\end{align*}
By the equation \eqref{PNLS-2}
\begin{align*}
\overline{w_t}&=-i\Delta\overline{w}+i|u|^4\overline{u}\\
&=-i\Delta\overline{w}+i|w|^4\overline{w}+i|u|^4\overline{u}-i|w|^4\overline{w}.
\end{align*}
From the above two equalities, we have
\begin{align}\label{3.9999999999}
M'(t)=&2\textrm{Im}\int_{\R^3}(\frac{x}{|x|}\cdot\nabla w+\frac{1}{|x|}w)(-i\Delta\overline{w})dx\nonumber\\
&+2\textrm{Im}\int_{\R^3}(\frac{x}{|x|}\cdot\nabla w+\frac{1}{|x|}w)(i|w|^4\overline{w})dx\nonumber\\
&+2\textrm{Im}\int_{\R^3}(\frac{x}{|x|}\cdot\nabla w+\frac{1}{|x|}w)(i|u|^4\overline{u}-i|w|^4\overline{w})dx\nonumber\\
:&=I_1+I_2+I_3.
\end{align}
For $I_1$, by integration-by-parts, we have
\begin{align*}
I_1=\int_{\R^3}\frac{1}{|x|}\big(|\nabla w|^2-\big|\frac{x}{|x|}\cdot\nabla w\big|^2\big)dx+2\pi |w(t,0)|^2\geq0.
\end{align*}
For $I_2$, by integration-by-parts again, we have
\begin{align*}
I_2=\frac{2}{3}\int_{\R^3}\frac{|w(t,x)|^6}{|x|}dx.
\end{align*}
For $I_3$, by H\"{o}lder's inequality and Lemma \ref{hardy}, we have
\begin{align*}
|I_3|&\lesssim\Big|\textrm{Im}\int_{\R^3}(\frac{x}{|x|}\cdot\nabla w+\frac{1}{|x|}w)(i|u|^4\overline{u}-i|w|^4\overline{w})dx\Big|\\
&\lesssim\|\nabla w\|_{L_x^2(\R^3)}\big\||u|^4u-|w|^4w\big\|_{L_x^2(\R^3)}.
\end{align*}
Hence, by the above three estimates and integrating in time in \eqref{3.9999999999}, we can obtain
\begin{align}\label{3.10}
\int_{I}\int_{\R^3}\frac{|w(t,x)|^6}{|x|}dxdt\lesssim {\sup_{t\in I}}M(t)+\|\nabla w\|_{L_t^{\infty}L_x^2(I\times\R^3)}\big\||u|^4u-|w|^4w\big\|_{L_t^1L_x^2(I\times\R^3)}.
\end{align}
For the first term, by \eqref{1.2323232}, \eqref{1.88} and H\"{o}lder's inequality
\begin{align}\label{3.11}
{\sup_{t\in I}}M(t)\lesssim\|w\|_{L_t^{\infty}L_x^{2}(I\times{\R^3})}\|w\|_{L_t^{\infty}\dot{H}_x^1(I\times{\R^3})}\lesssim N^{3(1-s_0)}\|w\|_{X_N(I)}.
\end{align}
Next, for the term $\big\||u|^4u-|w|^4w\big\|_{L_t^1L_x^2(I\times\R^3)}$, noting that $u=w+v$, by H\"{o}lder's inequality we have
\begin{align*}
\big\||u|^4u-|w|^4w\big\|_{L_t^1L_x^2(I\times\R^3)}&\lesssim \big\||u+w|^4|u-w|\big\|_{L_t^1L_x^2(I\times\R^3)}\\
&\lesssim\big\|(|w|^4+|v|^4)|v|\big\|_{L_t^1L_x^2(I\times\R^3)}\\
&\lesssim\|v\|_{L_t^5L_x^{10}(I\times\R^3)}^5+\|v\|_{L_t^2L_x^{\infty}(I\times\R^3)}\|w\|_{L_{t,x}^{8}(I\times\R^3)}^4.
\end{align*}
For the estimates about $v$, by Lemma \ref{linear estimate}, interpolation inequality and \eqref{1.12}, we have
\begin{align*}
\|v\|_{L_t^2L_x^{\infty}(I\times\R^3)}\lesssim N^{-s_0+}\delta_0,
\end{align*}
and
\begin{align*}
\|v\|_{L_t^5L_x^{10}(I\times\R^3)}\lesssim \|v\|_{L_t^2L_x^{\infty}(I\times\R^3)}^{\frac 2 5}\|v\|_{L_t^{\infty}L_x^6(I\times\R^3)}^{\frac 3 5}
\lesssim  N^{-s_0+\frac{1}{5}+}\delta_0.
\end{align*}
Further, by \eqref{1.2323232} and the above three estimates, we obtain that
\begin{align}\label{3.8888888}
\big\||u|^4u-|w|^4w\big\|_{L_t^1L_x^2(I\times\R^3)}&\lesssim  N^{-5s_0+1+}\delta_0^5+N^{-s_0+}\delta_0N^{\frac{9}{2}(1-s_0)}\|w\|_{X_N(I)}^4\nonumber\\
&\lesssim N^{-\frac{11}{2}s_0+\frac{9}{2}+}\delta_0\|w\|_{X_N(I)}^4\nonumber\\
&\lesssim \delta_0\|w\|_{X_N(I)}^4.
\end{align}
Hence, by \eqref{1.2323232}, \eqref{3.10}, \eqref{3.11} and \eqref{3.8888888}, we have
\begin{align*}
\int_{I}\int_{\R^3}\frac{|w(t,x)|^6}{|x|}dxdt\lesssim N^{3(1-s_0)}(\|w\|_{X_N(I)}+\delta_0\|w\|_{X_N(I)}^5).
\end{align*}
This finishes the proof.
\end{proof}
\end{lem}

From the above lemma we have the following result.

\begin{cor}\label{same timespace}
Under the same assumptions as in Lemma \ref{Morawetz}, then
\begin{align*}
\|w\|_{L_{t,x}^{8}(I\times{\R^3})}\lesssim N^{\frac{9}{8}(1-s_0)}\big(\|w\|_{X_N(I)}^{\frac{3}{8}}+\delta_0^{\frac{1}{8}}\|w\|_{X_N(I)}^{\frac{7}{8}}\big).
\end{align*}
\begin{proof}
By H\"{o}lder's inequality,
\begin{align*}
\int_{I}\int_{\R^3}|w(t,x)|^8dxdt=&\int_{I}\int_{\R^3}\frac{|w|^6}{|x|}|x||w|^2dxdt\\
\lesssim& \big\||x|^{\frac{1}{2}}w\big\|_{L_{t,x}^{\infty}(I\times{\R^3})}^2 \int_{I}\int_{\R^3}\frac{|w(t,x)|^6}{|x|}dxdt.
\end{align*}
By Lemma \ref{radial Sobolev} and \eqref{1.2323232}, we have
\begin{align*}
\big\||x|^{\frac{1}{2}}w\big\|_{L_{t,x}^{\infty}(I\times{\R^3})}\lesssim\|w\|_{L_t^{\infty}\dot{H}_x^1(I\times{\R^3})}\lesssim N^{3(1-s_0)}\|w\|_{X_N(I)}.
\end{align*}
Hence, by Lemma \ref{Morawetz} and the above estimates, we obtain
\begin{align*}
\int_{I}\int_{\R^3}|w(t,x)|^8dxdt\lesssim N^{9(1-s_0)}(\|w\|_{X_N(I)}^3+\delta_0\|w\|_{X_N(I)}^7),
\end{align*}
which gives the desired estimate.
\end{proof}
\end{cor}

\subsection{Energy estimate}

\begin{lem}\label{Energy eatimate}
Let $\frac{5}{6}<s_0<1$ and $\|w\|_{X_N(I)}\geq 1$, then
\begin{align*}
\|w\|_{L_t^{\infty}\dot{H}_x^1(I\times{\R^3})}\lesssim N^{3(1-s_0)}\big(1+\delta_0^{\frac{1}{2}}\|w\|_{X_N(I)}^{\frac{5}{2}}\big).
\end{align*}
\end{lem}
\begin{proof}
For simplicity, we denote $I=[0,T)$ and for any $t \in I$,
let
\begin{align*}
\widetilde{E}(t):=\frac{1}{2}\int_{\R^3}|\nabla w(t,x)|^2dx+\frac{1}{6}\int_{\R^3}|u(t,x)|^6dx.
\end{align*}
Taking product with $w_t$ on the equation \eqref{PNLS} and integration-by-parts, we have
\begin{align*}
\frac{d}{dt}\widetilde{E}(t)=\textrm{Re}\int_{\R^3}|u|^4u \overline{v_t}dx.
\end{align*}
Since $v$ is a linear solution, we further obtain
\begin{align*}
\frac{d}{dt}\widetilde{E}(t)=\textrm{Im}\int_{\R^3}|u|^4u\Delta \bar{v}dx.
\end{align*}
Integrating the above equality in time from $t_0$ to $t$, we get
\begin{align}\label{2.990000}
\widetilde{E}(t)=\widetilde{E}(t_0)+\textrm{Im}\int_{t_0}^t\int_{\R^3}|u|^4u\Delta \bar{v}dxdt',
\end{align}
where $t_0$ will be determined later.

For the first term $\widetilde{E}(t_0)$ in \eqref{2.990000}, by Sobolev inequality, we have
\begin{align}\label{2.19999}
\widetilde{E}(t_0)&\lesssim \|w(t_0)\|_{{\dot{H}_x^1}({\R^3})}^2+\|u(t_0)\|_{{L_x^6}(\R^3)}^6\nonumber\\
&\lesssim \| \nabla w\|_{L_t^{\infty}L_x^2([0,t_0]\times{\R^3})}^2+\| \nabla w\|_{L_t^{\infty}L_x^2([0,t_0]\times{\R^3})}^6+\|v\|_{L_t^{\infty}L_x^6([0,t_0]\times{\R^3})}^6.
\end{align}
Now, we estimate the term $ \| \nabla w\|_{L_t^{\infty}L_x^2([0,t_0]\times{\R^3})}$.
First of all, by Lemma \ref{Strichartz estimate} and Lemma \ref{linear estimate}, we have
\begin{align*}
\|\nabla e^{it\Delta}u_0\|_{L_t^{2}L_x^{6}(I\times{\R^3})}&\lesssim \|\nabla e^{it\Delta}(P_{\leq 1}f+P_{\geq 1} \chi_{\leq 1}f)\|_{L_t^{2}L_x^{6}(I\times{\R^3})}+\|\nabla e^{it\Delta}(P_{\geq 1} \chi_{\geq 1}f)_{out}\|_{L_t^{2}L_x^{6}(I\times{\R^3})}\\
&\lesssim \|P_{\leq 1}f+P_{\geq 1} \chi_{\leq 1}f\|_{\dot{H}^1(\R^3)}+\|P_{\geq 1} \chi_{\geq 1}f\|_{H^{s_0}(\R^3)}\\
&\lesssim\|\chi_{\leq 1}f\|_{H^1(\R^3)}+\|\chi_{\geq 1}f\|_{H^{s_0}(\R^3)}.
\end{align*}
Hence, given small constant $\widetilde{\delta_0}>0$, by choosing $t_0=t_0(u_0, \|\chi_{\leq 1}f\|_{H^1(\R^3)}+\|\chi_{\geq 1}f\|_{H^{s_0}(\R^3)})>0$ small enough, we have
\begin{align*}
\|\nabla e^{it\Delta}u_0\|_{L_t^{2}L_x^{6}([0, t_0]\times{\R^3})}\leq\widetilde{\delta_0}.
\end{align*}
Then using the standard fixed point argument, we can obtain
\begin{align}\label{2.880000}
\|\nabla u\|_{L_t^{2}L_x^{6}([0,t_0]\times{\R^3})}\lesssim \widetilde{\delta_0}.
\end{align}
By \eqref{1.12}, \eqref{2.880000}, Lemma \ref{Strichartz estimate} and Lemma \ref{linear estimate}, we have
\begin{align*}
\|\nabla w\|_{L_t^{\infty}L_x^{2}([0,t_0]\times{\R^3})}&\lesssim \|w_0\|_{{\dot{H}_x^1}({\R^3})}+\|\nabla (|u|^4u)\|_{L_t^{2}L_x^{\frac{6}{5}}([0,t_0]\times{\R^3})}\\
&\lesssim N^{1-s_0}+\|\nabla u\|_{L_t^{2}L_x^{6}([0,t_0]\times{\R^3})}\|u\|_{L_t^{\infty}L_x^{6}([0,t_0]\times{\R^3})}^4\\
&\lesssim N^{1-s_0}+\widetilde{\delta_0}\big(\|\nabla w\|_{L_t^{\infty}L_x^{2}([0,t_0]\times{\R^3})}^4+\|v\|_{L_t^{\infty}L_x^{6}([0,t_0]\times{\R^3})}^4\big)\\
&\lesssim N^{1-s_0}+\widetilde{\delta_0}\| \nabla w\|_{L_t^{\infty}L_x^{2}([0,t_0]\times{\R^3})}^4+\widetilde{\delta_0}N^{-4s_0+\frac{4}{3}+}\delta_0^4\\
&\lesssim N^{1-s_0}+\widetilde{\delta_0}\|\nabla w\|_{L_t^{\infty}L_x^{2}([0,t_0]\times{\R^3})}^4.
\end{align*}
Further, by the bootstrap argument,
\begin{align}\label{1228001}
\|\nabla w\|_{L_t^{\infty}L_x^{2}([0,t_0]\times{\R^3})}\lesssim N^{1-s_0}.
\end{align}
Hence, by \eqref{2.19999}, \eqref{1228001} and Lemma \ref{linear estimate}, we obtain
\begin{align}\label{1228002}
\widetilde{E}(t_0)\lesssim& N^{2(1-s_0)}+N^{6(1-s_0)}+N^{-6s_0+2+}\delta_0^6\nonumber\\
\lesssim& N^{6(1-s_0)}.
\end{align}
Next, we consider the second term in \eqref{2.990000}, by integration-by-parts, we have
\begin{align}\label{2.29999}
\textrm{Im}\int_{t_0}^t\int_{\R^3}|u|^4u\Delta \bar{v}dxdt'\lesssim& \int_{t_0}^t\int_{\R^3}|u|^4|\nabla u||\nabla v|dxdt'\nonumber\\
\lesssim&\int_{t_0}^T\int_{\R^3}|u|^4|\nabla w||\nabla v|dxdt'+\int_{t_0}^T\int_{\R^3}|u|^4|\nabla v|^2dxdt'\nonumber\\
:=&I_1+I_2.
\end{align}
We consider the term $I_1$ firstly. By  \eqref{1.12}, Lemma \ref{linear estimate} and interpolation inequality, we have
\begin{align}\label{2.110000}
\|\nabla v\|_{L_t^{2}L_x^{\infty}([t_0,T]\times{\R^3})}\lesssim N^{-s_0+\frac{1}{2}+}\delta_0,
\end{align}
and
\begin{align}\label{2.110001}
\| v\|_{L_{t,x}^{8}([t_0,T]\times{\R^3})}\lesssim \|v\|_{L_t^2L_x^{\infty}(I\times\R^3)}^{\frac 1 4}\|v\|_{L_t^{\infty}L_x^6(I\times\R^3)}^{\frac 3 4}
\lesssim  N^{-s_0+\frac{1}{4}+}\delta_0.
\end{align}
Further, by using \eqref{1.2323232}, \eqref{2.110000} and \eqref{2.110001}, we obtain
\begin{align}\label{2.49999}
I_1&\lesssim\|\nabla v\|_{L_t^{2}L_x^{\infty}([t_0,T]\times{\R^3})}\|\nabla w\|_{L_t^{\infty}L_x^{2}(I\times{\R^3})}\|u\|_{L_{t,x}^{8}(I\times{\R^3})}^4\nonumber\\
&\lesssim\|\nabla v\|_{L_t^{2}L_x^{\infty}([t_0,T]\times{\R^3})}\|\nabla w\|_{L_t^{\infty}L_x^{2}(I\times{\R^3})}\big(\|w\|_{L_{t,x}^{8}(I\times{\R^3})}^4+\|v\|_{L_{t,x}^{8}(I\times{\R^3})}^4\big)\nonumber\\
&\lesssim N^{-s_0+\frac{1}{2}+}\delta_0\cdot N^{3(1-s_0)}\|w\|_{X_N(I)}\cdot \big(N^{\frac{9}{2}(1-s_0)}\|w\|_{X_N(I)}^4+N^{-4s_0+1+}\delta_0^4\big)\nonumber\\
&\lesssim N^{-s_0+\frac{1}{2}+}\delta_0\cdot N^{3(1-s_0)}\|w\|_{X_N(I)}\cdot N^{\frac{9}{2}(1-s_0)}\|w\|_{X_N(I)}^4\nonumber\\
&\lesssim N^{-\frac{17}{2}s_0+8+}\delta_0\|w\|_{X_N(I)}^5\nonumber\\
&\lesssim N^{6(1-s_0)}\delta_0\|w\|_{X_N(I)}^5.
\end{align}
Next, we consider the term $I_2$. By \eqref{1.2}, \eqref{1.12}, \eqref{2.110000} and Lemma \ref{linear estimate}, we obtain
\begin{align}\label{2.59999}
I_2\lesssim& \|\nabla v\|_{L_t^{2}L_x^{\infty}([t_0,T]\times{\R^3})}^2\|u\|_{L_t^{\infty}L_x^{2}(I\times{\R^3})}\|u\|_{L_t^{\infty}L_x^{6}(I\times{\R^3})}^3\nonumber\\
\lesssim& N^{-2s_0+1+}\delta_0^2\cdot\big(N^{9(1-s_0)}\|w\|_{X_N(I)}^3+N^{-3s_0+1+}\delta_0^3\big)\nonumber\\
\lesssim& N^{-11s_0+10+}\delta_0^2\|w\|_{X_N(I)}^3\nonumber\\
\lesssim& N^{6(1-s_0)}\delta_0^2\|w\|_{X_N(I)}^3.
\end{align}
Hence, combining \eqref{2.990000}, \eqref{1228002}, \eqref{2.29999}, \eqref{2.49999} with \eqref{2.59999}, we can obtain
\begin{align*}
{\sup_{t\in I}}\widetilde{E}(t)&\lesssim N^{6(1-s_0)}\big(1+\delta_0\|w\|_{X_N(I)}^5+\delta_0^2\|w\|_{X_N(I)}^3\big)\\
&\lesssim N^{6(1-s_0)}\big(1+\delta_0\|w\|_{X_N(I)}^5\big).
\end{align*}
Further, from the definition of $\wt E(t)$, we have
\begin{align*}
\|w\|_{L_t^{\infty}\dot{H}_x^1(I\times{\R^3})}&\lesssim\big({\sup_{t\in I}}\widetilde{E}(t)\big)^{\frac{1}{2}}\\
&\lesssim N^{3(1-s_0)}\big(1+\delta_0^{\frac{1}{2}}\|w\|_{X_N(I)}^{\frac{5}{2}}\big).
\end{align*}
This completes the proof of this lemma.
\end{proof}

Now, we aim to prove Proposition \ref{priori prop}, which shows the assumption (H2) is valid.
\begin{proof}[Proof of Proposition \ref{priori prop}]
First, we show that for any $I$ such that $0 \in I \subset \R^+$,
\begin{align}\label{1228003}
\|w\|_{X_N(I)}\lesssim 1.
\end{align}
Indeed, if $\|w\|_{X_N(I)}\leq 1$, then \eqref{1228003} already holds.
Therefore, we can assume $\|w\|_{X_N(I)}\geq 1$.
Using Corollary \ref{same timespace}, Lemma \ref{Energy eatimate}, and Young's inequality, we obtain
\begin{align*}
\|w\|_{X_N(I)}&=N^{-3(1-s_0)}\|w\|_{L_t^{\infty}\dot{H}_x^1(I\times{\R^3})}+N^{-\frac{9}{8}(1-s_0)}\|w\|_{L_{t,x}^{8}(I\times{\R^3})}\\
&\le C( 1+\delta_0^{\frac{1}{2}}\|w\|_{X_N(I)}^{\frac{5}{2}}+\|w\|_{X_N(I)}^{\frac{3}{8}}+\delta_0^{\frac{1}{8}}\|w\|_{X_N(I)}^{\frac{7}{8}})\\
&\le C+C\delta_0^{\frac{1}{2}}\|w\|_{X_N(I)}^{\frac{5}{2}}+\frac{1}{4}\|w\|_{X_N(I)}+\delta_0^{\frac{1}{7}}\|w\|_{X_N(I)}\\
&\le C+C\delta_0^{\frac{1}{2}}\|w\|_{X_N(I)}^{\frac{5}{2}}+\frac{1}{2}\|w\|_{X_N(I)}.
\end{align*}
Then, we get
\begin{align*}
\|w\|_{X_N(I)}\lesssim 1+\delta_0^{\frac{1}{2}}\|w\|_{X_N(I)}^{\frac{5}{2}}.
\end{align*}
By the usual bootstrap argument, we obtain \eqref{1228003}. Hence, we finish the proof of Proposition \ref{priori prop}.
\end{proof}


\end{document}